\documentclass{amsart}
\usepackage{amssymb,amsmath,amsthm,enumerate,graphicx,mathtools,subcaption,tikz}

\newtheorem*{theorem*}{Theorem}
\newtheorem{theorem}{Theorem}
\newtheorem{lemma}[theorem]{Lemma}

\newtheorem{remark}{Remark}
\newtheorem*{remark*}{Remark}
\newtheorem{question}{Question}
\newenvironment{definition}[1][Definition]{\begin{trivlist}
\item[\hskip \labelsep {\bfseries #1}]}{\end{trivlist}}

\def\L{\mathcal{L}}
\def\R{\mathbb{R}}
\def\T{\mathbb{T}}
\def\M{\mathcal{M}}
\def\U{\mathcal{U}}
\def\S{\mathcal{S}}
\def\f{\varphi}
\def\e{\varepsilon}

\DeclareMathOperator\supp{supp}

\author[
  Andrew Clarke
]{
  Andrew Clarke\address{Andrew Clarke, \newline Universitat de Barcelona,  \newline Gran Via de Les Corts Catalanes 585, \newline Barcelona 08007, Spain}\email{  \texttt{andrew.clarke@ub.edu}}
}

\date{\today}

\title{Geodesics with Unbounded Speed on Fluctuating Surfaces}

\begin{document}
\maketitle
\begin{abstract}
We construct $C^{\infty}$ time-periodic fluctuating surfaces in $\R^3$ such that the corresponding nonautonomous geodesic flow has orbits along which the energy, and thus the speed, goes to infinity. 
\end{abstract}

\section{Introduction}

Consider a smooth manifold $M$ of dimension $n \geq 2$ with a smooth metric $g$. The geodesic flow on $TM$ is the flow associated with the Hamiltonian function $H (q,p) = \frac{1}{2} \, g (q) (p,p)$ with respect to the symplectic form $dq \wedge dp$, where $q \in M$ and $p \in T_q M$. Since this is an autonomous Hamiltonian system the energy $H$ is an integral of motion, and therefore so too is the speed $\sqrt{ 2 \, H}$. Suppose now we add a small nonautonomous time-periodic perturbation $\e \, \bar g$ to the metric to obtain $g_{\e}(q,t) = g (q) + \e \, \bar g (q,t)$ with $\bar g ( q, t + 2 \pi) = \bar g (q,t)$ for all $q \in M$, $t \in \R$, where $\e > 0$ is small. We say that $g_{\e}$ is a \emph{nonautonomous metric} on $M$ if $g_{\e} ( \cdot, t)$ defines a Riemannian metric on $M$ for each fixed value of time $t$. Nonautonomous metrics of the form $g_{\e}$ give rise to nonautonomous Hamiltonian functions $H_{\e} (q,p,t) = H(q,p) + \e \, \bar H (q,p,t)$ where $H$ is as before, and where $\bar H = \frac {1}{2} \bar g (q,t)(p,p)$. Since $H_{\e}$ is nonautonomous, its flow $\phi_{\e,t}$ no longer preserves the energy $H_{\e}$, and so a natural question arises: 
\begin{question}\label{question_metrics}
Does the time-periodic nature of $\bar g$ imply that the changes in energy cannot accumulate, or can we find $(q,p) \in TM$ such that $H_{\e} \left( \phi_{\e,t} (q,p),t \right) \to \infty$ as $t \to \infty$?
\end{question}
Orbits of this type, should they exist, travel progressively faster around the manifold, where the speed goes to infinity with time. 

A closely related question can be formulated as follows. Suppose now that $M$ is a smooth closed hypersurface of $\R^d$ where $d \geq 3$, and suppose $M$ is equipped with the induced metric $\delta|_{M}$ where $\delta$ is the standard Euclidean metric on $\R^d$ (i.e. $\delta$ is the $d \times d$ identity matrix). Consider now a $C^{\infty}$ $2 \pi$-time periodic embedding $G_{\e,t} : M \to \R^d$ such that $G_{\e,t}$ is $O(\e)$-close to the identity in $C^{\infty}$. Then $M_{\e,t} = G_{\e,t} (M)$ is a $2 \pi$-time periodic \emph{fluctuating hypersurface} in $\R^d$ that is $O (\e)$ close to $M$ in $C^{\infty}$. The geodesic flow on $M_{\e,t}$ is defined by using the coordinates from $M$ via the embedding $G_{\e,t}$ and using the (nonautonomous) pullback metric $g_{\e}(\cdot,t)=\left( G_{\e,t} \right)^* \left( \left. \delta \right|_{M_{\e,t}}\right)$; then the geodesic flow on the fluctuating hypersurface $M_{\e,t}$ is the flow associated with the Hamiltonian function $H_{\e} (q,p,t) = \frac{1}{2} \, g_{\e} (q,t) (p,p)$. 
\begin{question}\label{question_surfaces}
Does the geodesic flow on the fluctuating hypersurface $M_{\e,t}$ have orbits along which the energy $H_{\e}$ goes to infinity?
\end{question}
Question \ref{question_metrics} is closely related to the Mather problem in the case where $M$ is a torus \cite{bolotin1999unbounded}. Questions \ref{question_metrics} and \ref{question_surfaces} are not equivalent, because there exist metrics on $M$ that cannot be obtained by pulling back the metric from another hypersurface in $\R^d$ (see for example \cite{clarke2019generic,stojanov1990bumpy,stojanov1993generic}). These questions, in their full generality, are very difficult to answer, for a number of reasons that are outlined below. The purpose of this paper is to find a class of surfaces $M \subset \R^3$ and embeddings $G_{\e,t} : M \to \R^3$ that yield a positive answer to Question \ref{question_surfaces} (and a fortiori to Question \ref{question_metrics}) in the sense of the following theorem. 
\begin{theorem}\label{theorem_main1}
Let $M$ be a $C^{\infty}$ closed surface in $\R^3$ equipped with the restriction $\left. \delta \right|_{M}$ of the Euclidean metric $\delta$ such that the corresponding geodesic flow has a hyperbolic periodic orbit and a transverse homoclinic. Then there exist $C^{\infty}$ $2 \pi$-time periodic embeddings $G_{\e,t} : M \to \R^3$ and points $(q_0,p_0), (q_1,p_1) \in TM$ such that 
\begin{equation}\label{eq_nonautdiffusive}
H_{\e} \left( \phi_{\e,t} (q_0,p_0),t \right) \to \infty \quad \mathrm{as} \quad t \to \infty
\end{equation}
and
\begin{equation}\label{eq_nonautoscillatory}
\liminf_{t \to \infty} H_{\e} \left( \phi_{\e,t} (q_1,p_1),t \right) = c < \infty, \quad \limsup_{t \to \infty} H_{\e} \left( \phi_{\e,t} (q_1,p_1),t \right) = \infty
\end{equation}
where $\phi_{\e,t}$ is the flow associated with the Hamiltonian function $H_{\e} (q,p,t) = \frac{1}{2} g_{\e} (q,t)(p,p)$ with $M_{\e,t} = G_{\e,t} (M)$, $g_{\e}(\cdot,t)=\left( G_{\e,t} \right)^* \left( \left. \delta \right|_{M_{\e,t}}\right)$, and $p \in T_q M$. Moreover the initial speed $J_j = \| p_j \|$ is of order 1. 
\end{theorem}

Orbits of the form \eqref{eq_nonautdiffusive} are called \emph{diffusive trajectories} whereas orbits of the form \eqref{eq_nonautoscillatory} are referred to as \emph{oscillatory trajectories}. The strategy of the proof of Theorem \ref{theorem_main1} is as follows. The assumptions on the geodesic flow on the unperturbed manifold $M$ (i.e. existence of a hyperbolic closed geodesic $\gamma$ and transverse homoclinic geodesic $\xi$) imply that the geodesic flow has a hyperbolic periodic orbit in every energy level, together with a transverse homoclinic. Indeed, this is because the dynamics of the geodesic flow is the same in every energy level, as a result of the Hamiltonian $H_{\e}$ being homogeneous of second degree in the momenta $p$. Therefore we can construct a normally hyperbolic invariant manifold $\Lambda$ in $TM$ consisting of the hyperbolic closed geodesic on an interval of energy levels. The cylinder $\Lambda$ is parametrised by an arclength parameter on $\gamma$ and the speed of the geodesic flow; since we are interested in the situation where the speed goes to infinity, the cylinder $\Lambda$ is noncompact. As a result of the existence of the transverse homoclinic geodesic $\xi$, the stable and unstable manifolds of $\Lambda$ have a transverse homoclinic intersection.

Since the Hamiltonian $H_{\e}$ is nonautonomous, it is convenient to introduce a new pair of variables $\theta \in \T$ and $\Theta \in \R$ and to define an autonomous Hamiltonian $H_{\e}^*$ depending on $q,p,\theta,\Theta$. Here the angle $\theta$ represents the $2 \pi$-periodic time dependence of $H_{\e}$, and $\Theta$ is a dummy variable that is symplectically conjugate to $\theta$. It is then apparent from the form of the Hamiltonian that the speed of the geodesic flow goes to infinity if and only if $\Theta \to - \infty$. This construction is carried out in Section \ref{section_unperturbed}. 

In Section \ref{section_perturbations} we show how to compute explicitly the first order of the scattering map for a specific class of perturbations. This is done by passing to Fermi coordinates (see for example \cite{klingenberg1976lectures}) in a neighbourhood of a transverse homoclinic point, and designing a time-periodic nonautonomous perturbation of the metric that is localised in this neighbourhood. We then show that such perturbations of the metric can be realised by a fluctuating surface. 

Finally, in Section \ref{section_transitionchains} we choose a specific nonautonomous perturbation of the surface with zero time average, and prove that it produces orbits satisfying \eqref{eq_nonautdiffusive} and \eqref{eq_nonautoscillatory} using shadowing results from \cite{gidea2006topological}. 

This phenomenon is a form of \emph{Arnold diffusion} \cite{Arnold:1964,bolotin1999unbounded,delshams2000geometric}, in the \emph{a priori chaotic} case \cite{delshams2000geometric,gelfreich2017arnold}, meaning that the unperturbed system (i.e. the geodesic flow on the static manifold $M$) has a normally hyperbolic invariant manifold, the stable and unstable manifolds of which have a transverse homoclinic intersection. Other forms of Arnold diffusion have been studied: in the \emph{a priori unstable} case \cite{delshams2006geometric,treschev2004evolution}, the unperturbed system has a normally hyperbolic invariant manifold, but the stable and unstable manifolds coincide; and in the \emph{a priori stable} case the unperturbed system is completely integrable, and no further assumptions are made \cite{Bernard08,Kaloshin:2016,Cheng:2017,Kaloshin:2020}. In addition to the geometric approaches to the problem of Arnold diffusion, variational methods have also had success \cite{berti2003drift,berti2002functional,bessi1996approach,cheng2004existence,cheng2009arnold,mather1993variational}. 

This idea is closely related to \emph{Fermi acceleration} \cite{bolotin1999unbounded,delshams2008geometric2,gelfreich2008unbounded}. Fermi acceleration describes the unbounded growth of energy in nonautonomous Hamiltonian systems, typically systems with impacts (i.e. billiards; see \cite{dettmann2018splitting,gelfreich2011robust,gelfreich2012fermi}). In order to display either Arnold diffusion or Fermi acceleration, a Hamiltonian system must have at least two and a half degrees of freedom, because two-dimensional invariant tori of a Hamiltonian flow divide a four-dimensional phase space into invariant regions. For this reason, Theorem \ref{theorem_main1} addresses only the case where $M$ is a surface in $\R^3$, as the resulting nonautonomous geodesic flow has two and a half degrees of freedom, but the proof applies equally to the case of hypersurfaces in higher-dimensional Euclidean spaces. 

The setting of perturbations of geodesic flows in Arnold diffusion is well-studied \cite{clarke2022arnold,delshams2000geometric,delshams2006orbits,gidea2017perturbations}. However none of the previous literature applies directly to geodesic flows on fluctuating hypersurfaces and, to the best of the author's knowledge, the present article is the first in which such models are considered. Moreover the equations of these nonautonomous geodesic flows appear to be immune to an attack of the type contained in those papers. It is therefore unclear a priori that there is not some peculiarity of the nonautonomous geodesic flow equations on fluctuating hypersurfaces that prevents the unbounded drift of energy; the results of this paper prove that there is no such peculiarity. 

In this paper we establish the existence of orbits of unbounded energy only for a special class of locally supported nonautonomous perturbations of a special class of surfaces, but Questions \ref{question_metrics} and \ref{question_surfaces} in their full generality would pertain to \emph{generic} nonautonomous perturbations of \emph{generic} surfaces. Broadly, the difficulty in answering these questions from the generic point of view is twofold:
\begin{itemize}
\item
\textbf{The Unperturbed Problem:} To apply the usual methods of (a priori unstable or chaotic) Arnold diffusion, we need the unperturbed system to have a normally hyperbolic invariant manifold, with the stable and unstable manifolds having some (non-transverse or transverse) intersection. This amounts to the existence of a hyperbolic closed geodesic and a homoclinic geodesic for the unperturbed problem; the existence of these objects is an active area of research. It is known that for a fixed manifold $M$ there is a $C^2$-dense open set of $C^{\infty}$ metrics such that these objects exist \cite{contreras2010geodesic,contreras2002genericity}. Moreover it is known that there is a $C^{\infty}$-dense open set of metrics on $\mathbb{S}^2$ such that these conditions are satisfied \cite{knieper2002c}. It is also known that similar results hold for hypersurfaces of $\R^d$ in the analytic topology \cite{clarke2019generic}. However the question of $C^{\infty}$-genericity in the general case remains open. 
\item
\textbf{Analysis of Perturbations:} A typical approach that is used in the perturbation theory of geodesic flows is to pass to Fermi coordinates (see \cite{klingenberg1976lectures} for a description) and perturb the metric in those coordinates (see for example \cite{anosov1983generic,clarke2019generic,contreras2010geodesic,contreras2002genericity,klingenberg1991generic,stojanov1990bumpy,stojanov1993generic}). The issue with this approach in relation to the current problem is that we must adapt the Fermi coordinates to a homoclinic geodesic (an infinite curve on the compact manifold $M$ that is likely to have self-intersections) in order to analyse the scattering map, and Fermi coordinates break down when the curve to which they are adapted has self-intersections. In this paper we consider a neighbourhood of a point on the homoclinic geodesic in which it has no self-intersections, pass to Fermi coordinates in this neighbourhood, and show how to compute the perturbed scattering map when the perturbation of the metric is supported in this neighbourhood. However it is not clear how this method could be adapted to the setting of globally-supported perturbations, nor is it clear that there is another system of coordinates in which the nonautonomous flow can be better understood. 
\end{itemize}

There are a few examples of perturbation theoretic results in geodesic flows in which Fermi coordinates are not used, and the analysis is of a global nature; see for example \cite{hofer2002pseudoholomorphic,hofer2003finite,knieper2002c}. In \cite{knieper2002c}, the geodesic flow on $\mathbb{S}^2$ is considered; Birkhoff proved that geodesic flows on 2-spheres can be reduced to a Poincar\'e map on a global surface of section that is diffeomorphic to a cylinder $\T \times [0,1]$, and so the analysis in \cite{knieper2002c} can use the classical theory of twist maps of the annulus. These techniques however do not apply to manifolds with a different topology to that of $\mathbb{S}^2$. The papers \cite{hofer2002pseudoholomorphic,hofer2003finite} are in fact about 3-dimensional Reeb flows, which includes the restriction of geodesic flows on 2-dimensional Riemannian manifolds to the unit tangent bundle, and so they make use of the methods of symplectic topology. It is currently unclear whether such methods could be applied to the problem in consideration. 

One setting in which further results of this type could likely be produced is in perturbations of geodesic flows on triaxial ellipsoids in $\R^3$. In a beautiful exposition in \cite{knieper1994surface}, it is pointed out that such geodesic flows have a pair of hyperbolic periodic orbits (corresponding to the same closed geodesic, traversed in opposite directions) connected by a non-transverse heteroclinic cycle. It is also shown in that paper via Melnikov-type theory that certain perturbations of the metric can split the separatrices. It is possible that such analysis could also be applied to the case of geodesic flows on nonautonomous perturbations of ellipsoids. 

In fact, the problem addressed in the current paper was motivated by the following question: does the geodesic flow on a \emph{breathing ellipsoid} (an ellipsoid that is allowed to vary periodically with time through a family of ellipsoids, so that for each fixed time the static manifold is an ellipsoid that is $O(\e)$ close to the original ellipsoid) in $\R^3$ possess orbits along which the energy goes to infinity? This question seems potentially intractable for two reasons: first of all, the Poincar\'e-Melnikov potential will be some type of elliptic integral, which would likely be quite difficult to interpret analytically; and in any case, it seems to be possible that the scattering map is trivial, in which case the usual techniques of diffusion do not apply. Indeed, this question is similar to the problem addressed in \cite{dettmann2018splitting}, in which they considered billiard dynamics in a breathing ellipse in $\R^2$. In that paper, the scattering map turned out to be trivial, and so a secondary perturbation was added in order to guarantee its well-definition. 

\subsection*{Acknowledgments}
This project received funding from the European Research Council (ERC) under the European Union's Horizon 2020 research and innovation programme (grant agreement No. 757802). The author is grateful to Marcel Guardia and Dmitry Turaev for useful discussions.

\section{Analysis of the Unperturbed Problem}\label{section_unperturbed}

The purpose of this section is to establish the existence of a normally hyperbolic invariant manifold with a transverse homoclinic intersection, and to transform our nonautonomous Hamiltonian system with two and a half degrees of freedom to an autonomous Hamiltonian system with three degrees of freedom. In fact, the contents of this section are presented for a hypersurface $M$ of $\mathbb{R}^{n+1}$ where $n \geq 2$. The analysis in this section is similar to Section 3 of \cite{delshams2006orbits}.

\subsection{The A Priori Chaotic Setting}

Denote by $M$ a smooth closed manifold of dimension $n \geq 2$, and by $g$ a Riemannian metric on $M$. Suppose $M$ is a hypersurface of $\R^{n+1}$, and $g$ is the restriction to $M$ of the standard inner product on $\R^{n+1}$. Suppose there is a hyperbolic closed geodesic $\gamma$ on $M$; that is, there is a closed geodesic $\gamma$ on $M$ that gives rise to a hyperbolic periodic orbit of the geodesic flow in each energy level. Suppose moreover there is a transverse homoclinic geodesic $\xi$ to $\gamma$; that is to say, a geodesic curve $\xi$ on $M$ such that the corresponding orbit of the geodesic flow lies in the transverse homoclinic intersection of the stable and unstable manifolds of $\gamma$ in each energy level. 

Denote by $z=(z_0, \ldots, z_n)$ coordinates on $\R^{n+1}$, and by $v=(v_0, \ldots, v_n)$ the corresponding tangent coordinates. Then the restriction to $M$ of the 2-form $\omega = dz \wedge dv$ defines a symplectic form on $\M = TM$. Reparametrising $\R^{n+1}$ if necessary, we may assume without loss of generality that the length of $\gamma$ is $2 \pi$. We can thus give the geodesic $\gamma$ an arclength parametrisation, which for convenience we denote by $\gamma : \T \to \R^{n+1}$, where $\T = \R / 2 \pi \mathbb{Z}$, such that $\| \gamma' (\f) \| \equiv 1$. Then $\gamma (\f) \in M$ and $\gamma' (\f) \in T_{\gamma (\f)} M$ for all $\f \in \T$. Fix some $J_0 > 0$. The set
\[
\Lambda = \left\{ (z,v) \in \R^{2(n+1)} : z = \gamma (\f), \, v = J \gamma' (\f) \text{ where } \f \in \T, \, J \geq J_0 \right\}
\]
defines a 2-dimensional non-compact normally hyperbolic invariant cylinder of the geodesic flow on $M$ with respect to the metric $g$. Moreover $\Lambda$ is diffeomorphic to $\R \times [J_0, \infty)$, and is equipped with natural coordinates $(\f, J)$. In these coordinates, the vector field of the geodesic flow takes the form
\begin{equation}\label{eq_innervf}
\dot \f = J, \quad \dot J=0. 
\end{equation}

\begin{lemma}\label{lemma_innersympform}
The restriction to $\Lambda$ of the symplectic form $\omega$ is
\[
\left. \omega \right|_{\Lambda} = d \f \wedge d J. 
\]
\end{lemma}

\begin{proof}
On $\Lambda$ we have $z = \gamma (\f)$ and $v = J \gamma' (\f)$. Therefore $dz = \gamma' (\f) \, d \f$ and $dv = \gamma'(\f) \, dJ + J \gamma'' (\f) \, d \f$. It follows that $dz \wedge dv = d \f \wedge dJ$ since $\gamma' (\f) \cdot \gamma' (\f) = \| \gamma' (\f) \|^2 = 1$. 
\end{proof}

Now, we can also give the transverse homoclinic curve $\xi$ an arclength parametrisation $\xi: \R \to \R^{n+1}$, so $\| \xi' (x_0) \| \equiv 1$ for all $x_0 \in \R$. Fix some $\delta_0 \in \left(0, \pi \right)$, and define the set
\[
\Gamma = \left\{ (z,w) \in \R^{n+1} : z= \xi (x_0), \, v = J \xi' (x_0) \text{ where }  |x_0| \leq \delta_0, \, J \geq J_0  \right\}. 
\]
Then $\Gamma$ is a homoclinic channel for $\Lambda$, because the projections $\left. \pi^{s,u} \right|_{\Gamma} : \Gamma \to \Lambda$ are diffeomorphisms onto their respective images. We can therefore define the scattering map
\[
s = \pi_s \circ \left( \left. \pi_u \right|_{\Gamma} \right)^{-1} : \pi_u \left(\Gamma \right) \longrightarrow  \pi_s \left(\Gamma \right). 
\]
It can easily be proved that there are $a_{\pm} \in \R$ (independent of $x_0$ and $J$) such that $\pi_s \left(x_0,J\right) = (x_0 + a_+,J)$ and $\pi_u \left(x_0,J\right) = (x_0 + a_-,J)$. It follows that the scattering map is given by $s(\f, J) = (\f + \Delta, J)$ where $\Delta = a_+ - a_-$. 

\subsection{The Extended Phase Space}

Now, denote by $q$ coordinates on the manifold $M$, by $p$ the corresponding tangent coordinates, and by $g$ the Riemannian metric. Recall that the geodesic flow is the flow $\phi_t$ associated with the Hamiltonian function $H(q,p) = \frac{1}{2} g(q) (p,p)$. We now introduce a new variable $\theta \in \T$, that represents time, and its symplectic conjugate $\Theta \in \R$. Denote by $\M^* = \M \times \T \times \R$ the \emph{extended phase space}, and by 
\begin{equation}\label{eq_extendedsymp} 
\omega^* = \omega + d \theta \wedge d \Theta
\end{equation}
the \emph{extended symplectic form} on $\M^*$ where $\omega$ is the symplectic form on $\M = TM$. Define the \emph{extended Hamiltonian} $H^* : \M^* \to \R$ by
\begin{equation} \label{eq_extendedunperturbedham}
H^*(q,p,\theta,\Theta) = H (q,p) + \Theta. 
\end{equation}
Denote by $\phi^*_t$ the flow associated with this Hamiltonian. For $(q,p,\theta, \Theta) \in \M^*$ we have $\phi^*_t (q,p,\theta, \Theta) = (\phi_t(q,p), \theta + t, \Theta)$. It follows that the manifold 
\[
\Lambda^* = \left\{ ( q,p,\theta, \Theta) \in \M^* : (q,p) \in \Lambda \right\}
\]
defines a normally hyperbolic invariant manifold for $\phi^*_t$, and is equipped with natural coordinates $(\f,J,\theta, \Theta)$. Moreover, the manifold 
\[
\Gamma^* = \left\{ ( q,p,\theta, \Theta) \in \M^* : (q,p) \in \Gamma \right\}
\]
is a homoclinic channel, and so we have a scattering map $s^* = \pi_s^* \circ \left( \left. \pi_u^* \right|_{\Gamma^*} \right)^{-1}: \pi_u^* (\Gamma^*) \to \pi_s^* (\Gamma^*)$ where $\pi^*_{s,u} : W^{s,u} (\Lambda^*) \to \Lambda^*$ are the holonomy maps related to the normally hyperbolic invariant manifold $\Lambda^*$. It is not hard to see that $s^*$ is given by 
\begin{equation}\label{eq_unperturbedscattering}
s^*( \f ,J,\theta,\Theta) = (\f + \Delta, J, \theta, \Theta). 
\end{equation}
It is clear that Theorem \ref{theorem_main1} is equivalent to the following theorem. 
\begin{theorem}\label{theorem_main2}
There are $C^{\infty}$ time-periodic embeddings $G_{\e,t} : M \to \R^3$ for which there exists $(q_0,p_0,\theta_0,\Theta_0)$, $(q_1,p_1,\theta_1,\Theta_1) \in \M^*$ such that 
\[
\Theta_0 (t) \to - \infty \quad \mathrm{as} \quad t \to \infty
\]
and
\[
\liminf_{t \to \infty} \Theta_1 (t) = - \infty, \quad \limsup_{t \to \infty} \Theta_1 (t) = c^* \in \R
\]
where $\phi^*_{\e,t} (q_j,p_j,\theta_j,\Theta_j) = (q_j(t),p_j(t),\theta_j(t),\Theta_j(t))$ with $\phi^*_{\e,t}$ denoting the flow of the Hamiltonian function 
\[
H_{\e}^* (q,p,\theta, \Theta) = g_{\e}(q,\theta) (p,p) + \Theta
\] 
where $g_{\e}(\cdot,t)=\left( G_{\e,t} \right)^* \left( \left. \delta \right|_{M_{\e,t}}\right)$. Moreover the initial speed $J_j = \| p_j \|$ is of order 1. 
\end{theorem}

\begin{remark}
Results of this type typically require initial speed or energy to be very large (see for example \cite{delshams2006orbits,dettmann2018splitting}); however the perturbative technique used in this paper (see Lemma \ref{lemma_scattering}) allows the starting speed to be $O(1)$. 
\end{remark}

\section{Nonautonomous Perturbations of the Metric}\label{section_perturbations}

We now assume that $M$ is a smooth closed surface in $\R^3$ (but this construction can easily be extended to higher dimension), with a Riemannian metric $g$ defined by the restriction of the standard inner product on $\R^3$ to $M$. Moreover we assume that $M$ has a hyperbolic closed geodesic $\gamma$ with a transverse homoclinic geodesic $\xi$. We now introduce a $C^{\infty}$ nonautonomous perturbation $\bar g : M \times \R \to \R^{2 \times 2}$ that is $2 \pi$-periodic in time, and consider the perturbed metric
\[
g_{\e} = g + \e \bar g
\]
for small $\e > 0$. Denote by $H_{\e} = H + \e \bar H$ the perturbed Hamiltonian of the geodesic flow, where $\bar H (q,p,t) = \bar{g} (q,t)(p,p)$. Lifting this problem to the extended phase space, we obtain the Hamiltonian $H_{\e}^* : \M^* \to \R$ given by $H_{\e}^* =H^* + \e \bar H^*$ where $H^*$ is defined by \eqref{eq_extendedunperturbedham}, and the Hamiltonian $\bar H^*$ is 
\begin{equation}\label{eq_hampert}
\bar H^* (q,p,\theta,\Theta)= \frac{1}{2} \bar g (q, \theta) (p,p). 
\end{equation}
Denote by $\phi^*_{\e, t} : \M^* \to \M^*$ the flow of the Hamiltonian $H^*_{\e}$.

\subsection{Computation of the Perturbed Scattering Map}\label{eq_scatteringmapcomp}

Assume that $\bar g$ is supported away from the hyperbolic closed geodesic $\gamma$. Then the manifold $\Lambda^*$ remains a normally hyperbolic invariant manifold for the perturbed flow $\phi^*_{\e, t}$. Since the perturbation is small, the homoclinic channel persists as a manifold $\Gamma^*_{\e}$ that is $O(\e)$-close to $\Gamma^*$, and so we have a perturbed scattering map $s_{\e}^*$ defined on a subset of $\Lambda^*$. Define the vector field $\S^*_{\e}$ on $\pi^*_{\epsilon,s}(\Gamma_{\e}^*)$ by the equation $\frac{d}{d \e} s^*_{\e} = \S^*_{\e} \circ s^*_{\e}$. Denote by $\Omega$ the standard $4 \times 4$ symplectic matrix. It was shown in \cite{delshams2008geometric} (see Theorem 32) that $\S^*_{\e} = \Omega \nabla S^*_{\e}$ where the Hamiltonian function $S^*_{\e}$ is given by
\begin{equation}\label{eq_hamofpertofscatteringmap}
S^*_{\e} (\f,J,\theta, \Theta) = \int_{- \infty}^{+\infty} \bar{H}^* \circ \phi^*_{\e,t} \circ \left( \left. \pi^*_{\e,s} \right|_{\Gamma^*_{\e}} \right)^{-1} (\f,J,\theta, \Theta) \, dt. 
\end{equation}
It follows that we can write 
\begin{equation}\label{eq_perturbedscatteringmap}
s^*_{\e} (\f,J,\theta, \Theta) = s^* (\f,J,\theta, \Theta) + \e \, \left( \Omega \nabla S^* \right) \circ s^*_{\e} (\f,J,\theta, \Theta) + O \left( \e^2 \right)
\end{equation}
with
\[
S^* (\f,J,\theta, \Theta) = S^*_0 (\f,J,\theta, \Theta) = \int_{- \infty}^{+\infty} \bar{H}^* \left( Z (\f,J,\theta, \Theta, t) \right) \, dt
\]
where
\begin{align}
Z (\f,J,\theta, \Theta, t) =& \phi^*_{t} \circ \left( \left. \pi^*_{s} \right|_{\Gamma^*} \right)^{-1} (\f,J,\theta, \Theta) = \phi^*_t \left( \xi (\f - a_+), J \xi' (\f - a_+), \theta, \Theta \right) \nonumber \\
=& \left( \xi (\f + t J - a_+), J \, \xi' (\f + t J - a_+), \theta + t, \Theta \right) \label{eq_homoclinicparametrisation1}
\end{align}
In what follows we compute $S^*$ explicitly for a certain class of perturbations $\bar g$.

Now, by translating the parametrisation of the homoclinic geodesic $\xi$ (and adjusting $a_{\pm}$ accordingly) if necessary, we may assume that $\xi (0)$ is not a point of self-intersection of the curve $\xi$, and does not intersect the curve $\gamma$. We can therefore take a neighbourhood $\U$ of $\xi (0)$ in $M$ such that:
\begin{itemize}
\item
$\U \cap \gamma = \emptyset$;
\item
$\U \cap \xi$ consists only of a small segment of $\xi$ containing $\xi(0)$ in its interior; and
\item
$\xi$ does not self-intersect inside $\U$. 
\end{itemize}
We can therefore introduce in $\U$ Fermi coordinates $x=(x_0,x_1)$ adapted to the geodesic $\xi$, defined as follows. Suppose there is some $\delta_0 > 0$ such that $\U \cap \xi = \xi ((-\delta_0,\delta_0))$; if there is not, we can simply shrink $\U$ slightly. Choose a vector $e_1 \in T_{\xi (0)} M$ such that $\| e_1 \| = 1$ and $g(\xi(0))(\xi'(0),e_1) = 0$. Obtain the vector $e_1 (x_0) \in T_{\xi(x_0)} M$ by parallel transporting $e_1$ along $\xi$. Therefore $\xi'(x_0)$, $e_1(x_0)$ form an orthonormal basis of $T_{\xi(x_0)} M$ for each $x_0 \in (- \delta_0, \delta_0)$. For some sufficiently small $\delta_1 > 0$, define the map $h : (-\delta_0, \delta_0) \times (\delta_1, \delta_1) \to \U$ by 
\begin{equation}\label{eq_fermitransf}
h (x) = \pi \circ \phi_1 ( \xi (x_0), x_1 \, e_1 (x_0))
\end{equation}
where $x=(x_0,x_1)$, $\pi : TM \to M$ is the canonical projection, and $\phi_1$ is the time-1 map of the geodesic flow on $TM$. By shrinking the neighbourhood $\U$ if necessary, we assume that $h$ is a diffeomorphism onto $\U$. Denote by $y=(y_0,y_1)$ the corresponding tangent coordinates. In Fermi coordinates $x=(x_0,x_1)$ the point $\xi (x_0)$ is given by $(x_0,0)$, and in the tangent coordinates $y=(y_0,y_1)$ the vector $J \, \xi'(0)$ is given by $(J,0)$. It follows that, in the coordinates $(x,y,\theta,\Theta)$, the trajectory $Z$ defined by \eqref{eq_homoclinicparametrisation1} takes the form
\begin{equation}\label{eq_homoclinicparametrisation2}
Z(\f,J,\theta,\Theta,t) = \left( \left( \f + t J - a_+,0 \right), \left( J,0 \right), \theta + t, \Theta \right)
\end{equation}
locally, when the $x_0$-component is near $\xi(0)$. 

Choose $\alpha \in C^{\infty} (\R)$ such that $\alpha (x_0) \geq 0$ for all $x_0 \in \R$, $\supp \, \alpha \subset (-1,1)$, and $\int_{-\infty}^{+ \infty} \alpha (x_0) \, d x_0=1$. For $\rho > 0$ define $\alpha_{\rho} (x_0)= \frac{1}{\rho} \alpha \left( \frac{x_0}{\rho} \right)$. Then $\alpha_{\rho} \in C^{\infty}(\R)$, $\supp \, \alpha_{\rho} \subset (-\rho,\rho)$, and $\int_{-\infty}^{+ \infty} \alpha_{\rho} (x_0) \, d x_0=1$. Moreover $\alpha_{\rho}$ approximates the Dirac delta function in the limit as $\rho \to 0$ in the sense of distributions (see Lemma \ref{lemma_scattering}). Later we will fix a value of $\rho$, but for now we simply assume that $\rho < \delta_0$. Choose $\beta \in C^{\infty} (\R)$ such that $\beta (0) = 1$ and $\supp \, \beta \subset (- \delta_1, \delta_1)$. We define the component $\bar{g}_{00} \in C^{\infty} (M \times \T)$ of $\bar g$ by
\begin{equation}\label{eq_metricpertstructure}
\bar{g}_{00}(x,\theta) = 2 \, \alpha_{\rho} (x_0) \, \beta (x_1) \, \chi (\theta)
\end{equation}
where $\chi$ is any function in $C^{\infty}(\T)$. By construction, we have $\supp \, \bar{g}_{00} \subset \U \times \T$. For now, the only assumption we make about $\bar{g}_{ij}$ where $i,j$ are not both 0 is that $\supp \bar{g}_{ij} \subset \U \times \T$. 

\begin{lemma}\label{lemma_scattering}
Assume the component $\bar g_{00}$ of the perturbation of the metric is of the form \eqref{eq_metricpertstructure}. Let $(\f, J, \theta, \Theta) \in \pi^*_{\e,u} (\Gamma^*_{\e})$, and denote by $(\bar \f, \bar J, \bar \theta, \bar \Theta) = s^*_{\e} (\f, J, \theta, \Theta)$ the image of the point $(\f, J, \theta, \Theta)$ under the perturbed scattering map $s^*_{\e}$. Then
\begin{equation}\label{eq_thetascattering}
\bar \Theta = \Theta - \e \, J \, \left( \chi' \left( \theta + \frac{a_- - \f}{J} \right) + R_{\rho}^0 (\f,J,\theta,\Theta) \right) + O \left( \e^2 J \right)
\end{equation}
and
\begin{equation}\label{eq_thetascatteringderivative}
\frac{\partial \bar \Theta}{\partial \theta} = \e \, J \, \left( \chi'' \left( \theta + \frac{a_- - \f}{J} \right) + R_{\rho}^1 (\f,J,\theta,\Theta) \right) + O \left( \e^2 J \right)
\end{equation}
where $R_{\rho}^j$ are smooth families of functions such that
\[
\lim_{\rho \to 0} R_{\rho}^j (\f,J,\theta,\Theta) = 0
\]
as long as $J \geq J_0$ for some $J_0 > 0$, and where the rate of convergence of the limit is independent of $\f,\theta,\Theta$, and depends only on the lower bound $J_0$ of $J$. 
\end{lemma}

\begin{proof}
By \eqref{eq_perturbedscatteringmap} we have
\begin{equation}\label{eq_scatteringfirstthetaformula}
\bar \Theta = \Theta - \e \, \frac{\partial S^*}{\partial \theta} (\f + \Delta, J, \theta, \Theta) + O \left( \e^2 \right). 
\end{equation}
Recall $Z(\f,J,\theta,\Theta)$ takes the form \eqref{eq_homoclinicparametrisation2} whenever the projection onto $M$ of $Z(\f,J,\theta,\Theta)$ lies in $\U$. Therefore since $\bar g$, and thus $\bar H^*$, is supported only when the $x$-component lies in $\U$, we can write
\begin{align*}
\bar H^* \left( Z(\f,J,\theta,\Theta, t) \right) =& \frac{1}{2} \bar g (( \f + t J - a_+,0), \theta + t)((J,0),(J,0)) \\
=& \frac{1}{2} \bar{g}_{00} (( \f + t J - a_+,0), \theta + t) J^2 \\
=& \alpha_{\rho} \left(\f + t J - a_+ \right) \beta(0) \chi (\theta + t) J^2 \\
=& \alpha_{\rho} \left(\f + t J - a_+ \right)  \chi (\theta + t) J^2,
\end{align*}
and so
\[
S^*(\f,J,\theta,\Theta) = \int_{-\infty}^{+\infty} \alpha_{\rho} (\f + t J - a_+) \chi(\theta + t) J^2 \, dt
\]
which implies that
\begin{equation}\label{eq_thetaderivativeofs}
\frac{\partial S^*}{\partial \theta} (\f + \Delta,J,\theta,\Theta) = J^2 \int_{-\infty}^{+\infty} \alpha_{\rho} (\f + t J - a_-) \chi'(\theta + t)  \, dt
\end{equation}
where we have used the fact that $\Delta = a_+ - a_-$. 

We now claim that 
\[
\lim_{\rho \to 0} J \int_{-\infty}^{+\infty} \alpha_{\rho} (\f + t J - a_-) \chi'(\theta + t)  \, dt = \chi' \left( \theta + \frac{a_- - \f}{J} \right) 
\]
where the rate of convergence depends only on the lower bound $J_0$ of $J$. Indeed, write $t^* = \frac{a_- - \f}{J}$, define
\[
I_{\rho} = \left| J \int_{- \infty}^{+ \infty} \alpha_{\rho} (\f + t J - a_-) \chi' (\theta + t) \, dt - \chi' (\theta + t^*) \right|,
\]
and choose some $\hat \e >0$. If we can find $\hat \rho >0$, depending only on $J_0$, $\hat \e$, such that $I_{\rho} \leq \hat \e$ for all $\rho \in (0,\hat \rho]$, then the claim is proved. Using the fact that the integral over $\R$ of $\alpha_{\rho}$ is 1, and making the substitution $u = \f + t J - a_-$, we see that
\[
I_{\rho} = \left| \int_{- \infty}^{+ \infty} \alpha_{\rho} (u) \left( \chi' \left(\theta + t^* + \frac{u}{J} \right) - \chi' (\theta + t^*) \right) \, du  \right|.
\]
Since $\chi \in C^{\infty} (\T)$, we can choose $C>0$ such that $\left| \chi' (\theta) \right| \leq C$ for all $\theta \in \T$, and we can choose $\sigma > 0$, independent of $\f, J, \theta, \Theta$, such that if $|u| \leq \sigma J$ then
\begin{equation}\label{eq_chiprimecontinuity}
\left| \chi' \left( \theta + t^* + \frac{u}{J} \right) - \chi' ( \theta + t^* ) \right| \leq \hat \e. 
\end{equation}
Therefore, using again the fact that the integral over $\R$ of $\alpha_{\rho}$ is 1, we obtain
\begin{align}
I_{\rho} \leq& \int_{| u | \leq \sigma J} \alpha_{\rho} (u) \, \left| \chi' \left( \theta + t^* + \frac{u}{J} \right) - \chi' ( \theta + t^* ) \right| \, du \nonumber \\
& \quad + \int_{| u | > \sigma J} \alpha_{\rho} (u) \, \left| \chi' \left( \theta + t^* + \frac{u}{J} \right) - \chi' ( \theta + t^* ) \right| \, du \nonumber \\
\leq& \, \hat{\e} + 2 \, C \, K_{\rho} \label{eq_integraldiffbound}
\end{align}
where
\[
K_{\rho} = \int_{| u | > \sigma J} \alpha_{\rho} (u) \,  du = \int_{| u | > \sigma J} \frac{1}{\rho} \, \alpha \left( \frac{u}{\rho} \right) \,  du = \int_{| v | > \frac{ \sigma J}{\rho}} \alpha ( v ) \, dv \leq \int_{| v | > \frac{ \sigma J_0}{\rho}} \alpha ( v ) \, dv
\]
since $\alpha $ is a positive function and since $J \geq J_0$, where we have made the substitution $v = \frac{u}{\rho}$. Choose any $\hat \rho \in (0, \sigma J_0]$. Clearly $\hat \rho$ depends only on $\hat \e$, $J_0$ since we can assume that $\sigma = \sigma \left( \hat \e \right)$ is chosen to be the largest number for which \eqref{eq_chiprimecontinuity} is satisfied for all $|u| \leq \sigma J$. Moreover, since $\alpha$ is a positive function, and since $\supp \alpha \subset (-1,1)$, for any $\rho \in (0, \hat \rho]$ we have
\[
K_{\rho} \leq \int_{| v | > \frac{ \sigma J_0}{\hat \rho}} \alpha (v) \, dv \leq \int_{| v | >1} \alpha (v) \, dv = 0. 
\]
Combining this with \eqref{eq_integraldiffbound}, we obtain $I_{\rho} \leq \hat \e$ for any $\rho \in (0, \hat \rho ]$, where $\hat \rho$ depends only on $J_0$, $\hat \e$, so the claim is proved. 

To complete the proof of formula \eqref{eq_thetascattering}, notice that, in order to compute the integral of the term of order $\e^k$ for $k \geq 2$, we will again have to make a substitution of the form $u = \f + t J - a_- + \cdots$, which yields $du = (J + \cdots) \, dt$ so we always lose a factor of $J$ from equation \eqref{eq_hamofpertofscatteringmap}. Therefore the error terms are indeed of order $\e^2 J$. Combining this fact with \eqref{eq_scatteringfirstthetaformula}, \eqref{eq_thetaderivativeofs}, and the claim proved above, we thus obtain \eqref{eq_thetascattering}, as required. 

In order to prove \eqref{eq_thetascatteringderivative} we simply differentiate \eqref{eq_thetaderivativeofs} with respect to $\theta$ and repeat the above procedure. 

\end{proof}

\begin{remark}
As Lemma \ref{lemma_scattering} will be used in the proof of Theorem \ref{theorem_main2}, we will be dealing with situations in which $| \Theta |$, $\| p \|$, and $J$ are very large. This can be delicate in perturbation-theoretic settings such as this one, where there are terms of order $\e$. The typical way of dealing with this is to rescale variables, the Hamiltonian, the symplectic form, and time, and to work with the rescaled system \cite{delshams2006orbits}. In our case, if $J = O ( \e^{-n} )$ for $n \in \mathbb{N}$ then we can rescale: variables via $\tilde p = \e^n p$, $\tilde J = \e^n J$, $\tilde \Theta = \e^n \Theta$; the Hamiltonian by $\tilde H_{\e} = \e^{2n} H_{\e}$; the symplectic forms by $\tilde \omega = \e^n \omega$ and $\tilde \omega^* = \e^n \omega^*$; and time by $\tilde t = \e^n t$. In this way we obtain the same dynamical system as the case when $J = O(1)$ due to the homogeneity of the Hamiltonian $H_{\e}$ in momenta. This proves the consistency of the results of Lemma \ref{lemma_scattering} when $\| p \|$ is large. 
\end{remark}

\subsection{Realising the Perturbed Metric as a Fluctuating Surface}

In this section we use results from \cite{clarke2019generic} to construct a fluctuating surface whose geodesic flow is equivalent to a nonatuonomous metric on the static unperturbed hypersurface whose first component is of the form \eqref{eq_metricpertstructure}. Suppose the hypersurface $M$ is of the form 
\[
M = \left\{ z \in \mathbb{R}^{n+1} : Q(z) = 0 \right\}
\]
for some $Q \in C^{\infty} \left( \R^{n+1} \right)$. Then the \emph{normal curvature} of $M$ in the direction $v \in T_z M$ is defined by 
\begin{equation}\label{eq_normalcurvature}
\kappa (z, v) = \sum_{i,j=0}^n \frac{\partial^2 Q}{\partial z_i \, \partial z_j} \left( z \right) \, v_i \, v_j. 
\end{equation}
The periodicity of the geodesic $\gamma$ implies that there is some $\f^* \in \T$ such that $\kappa \left( \gamma ( \f^* ), \gamma' ( \f^* ) \right) \neq 0$ \cite{clarke2019generic}. It follows that there is an open neighbourhood $V_0$ of $\left( \gamma ( \f^* ), \gamma' ( \f^* ) \right)$ in the unit tangent bundle $T^1 M$ such that $\kappa$ does not vanish on $V_0$. Since $\xi$ is homoclinic to $\gamma$, by translating its parametrisation if necessary, we can assume that $(\xi (0), \xi' (0)) \in V_0$. Shrinking $\delta_0$ if necessary (where $\delta_0$ was defined in Section \ref{eq_scatteringmapcomp}) we can thus assume that $(\xi (x_0), \xi'(x_0)) \in V_0$ for all $x_0$ satisfying $|x_0| \leq \delta_0$. 

Now, let $\psi \in C^{\infty} (\R^{n+1} \times \R)$ such that $\psi (z,t+2 \pi) = \psi (z,t)$ for all $(z,t) \in \R^{n+1} \times \R$ and such that $M \cap \, \supp \psi \subset \U$, where the set $\U$ was defined in Section \ref{eq_scatteringmapcomp}. Consider the $2 \pi$-time periodic fluctuating surface
\[
M_{\e,t} = \left\{ z \in \mathbb{R}^{n+1} : Q(z) + \e \, \psi (z,t)= 0 \right\}. 
\]
Denote by 
\begin{equation}\label{eq_unitnormal}
n(z) = - \frac{ \nabla Q(z)}{\| \nabla Q (z) \|}
\end{equation}
the inward-pointing unit normal to $M$ at $z$. The following lemma is equivalent to a lemma in \cite{clarke2019generic}. 
\begin{lemma}
There is a $C^{\infty}$ family of embeddings $G_{\e,t} : M \hookrightarrow \R^{n+1}$ such that
\[
G_{\e,t} (M) = M_{\e,t}, \quad G_{\e,t} (z) = z + \e \, \psi (z,t) \, n(z) + O ( \e^2)
\]
for all sufficiently small values of $\e$, where $n(z)$ is the unit normal, defined by \eqref{eq_unitnormal}. 
\end{lemma}
Denote by $\delta$ the standard metric on $\R^{n+1}$, and consider the pullback $g_{\e} (\cdot,t) =  \left(G_{\e,t} \right)^* \left( \left. \delta \right|_{M_{\e,t}} \right) $ of the restriction of $\delta$ to $M_{\e,t}$ under $G_{\e,t}$. Then $g_{\e}$ defines a nonautonomous metric on the static unperturbed manifold $M$. The following lemma is equivalent to a result from \cite{clarke2019generic}. 
\begin{lemma}
The Fermi coordinates $x$ on the subset $\U$ of $M$ define coordinates on a subset of $M_{\e,t}$ via the map $G_{\e,t} \circ h : (- \delta_0,\delta_0) \times (- \delta_1,\delta_1) : M \to M_{\e,t}$, and the pullback metric in Fermi coordinates is 
\[
g_{\e} (\cdot,t) =  \left(G_{\e,t} \circ h \right)^* \left( \left. \delta \right|_{M_{\e,t}} \right) = g ( \cdot ) + \e \, \bar g ( \cdot, t) + O(\e^2)
\]
where the map $h$ is defined by \eqref{eq_fermitransf}, $g$ is the unperturbed metric, and the components of $\bar g$ are
\begin{equation}\label{eq_fermisurfacepert}
\bar g_{ij} (x,t) = 2 \, \psi \left( h (x), t \right) \tilde C_{ij} (x)
\end{equation}
where $\tilde C_{ij}$ are $C^{\infty}$ functions of $x=(x_0,x_1, \ldots, x_n)$ such that 
\begin{equation}\label{eq_fermipertsurfcomps}
\tilde C_{00} (x_0, 0, \ldots, 0) = - \kappa \left( \xi (x_0), \xi' (x_0) \right)
\end{equation}
where the normal curvature $\kappa$ is defined by \eqref{eq_normalcurvature}.
\end{lemma}

Therefore if we define 
\begin{equation}\label{eq_surfaceperturbationformula}
\psi ( z, t ) = - \kappa \left( \xi \left( h_0^{-1} (z) \right), \xi' \left( h_0^{-1} (z) \right) \right)^{-1} \alpha_{\rho} \left( h_0^{-1} (z) \right) \, \beta \left( h_1^{-1} (z) \right) \, \chi (t)
\end{equation}
for $z \in \U$ where $h_j^{-1}$ are the components of the inverse of the diffeomorphism defined by \eqref{eq_fermitransf}, and extend $\psi$ smoothly to 0 outside a neighbourhood of $\U$ in $M$, then $\bar{g}_{00}$ is of the form \eqref{eq_metricpertstructure} along $\U$ as a result of \eqref{eq_fermisurfacepert} and \eqref{eq_fermipertsurfcomps}, and so the corresponding scattering map can be computed using Lemma \ref{lemma_scattering}.

\section{Construction of Transition Chains for a Specific Perturbation}\label{section_transitionchains}

The purpose of this section is to construct drifting trajectories for certain perturbations of the surface of the form \eqref{eq_surfaceperturbationformula}, where $\alpha_{\rho}$, $\beta$ are as described in Section \ref{eq_scatteringmapcomp}. For concreteness we consider the case when
\begin{equation}\label{eq_chiiscos}
\chi (t) = \cos t.
\end{equation}
Note however that the following argument equally applies to any nonconstant $\chi \in C^{\infty} \left( \T \right)$ with zero average. The idea is to apply the shadowing results of \cite{gidea2006topological}. Since the results of that paper are for maps, we first reduce the flow to a Poincar\'e map before checking the assumptions of the shadowing theorem.

In order to construct a Poincar\'e map, we extend the coordinates $( \f, J, \theta, \Theta)$ on $\Lambda^*$ to $C^r$ coordinates $( \f, J, \theta, \Theta, s, u)$ in a neighbourhood of $\Lambda^*$. Such coordinates are known to exist; for example see \cite{jones2009generalized} for the construction of such a coordinate system that straightens the stable and unstable manifolds $W^{s,u} \left( \Lambda^* \right)$. Therefore the section $\{ \f = 0 \}$ is well-defined in a neighbourhood of $\Lambda^*$. We now fix an energy level $\{ H_{\e}^* = E_0 \}$ of the Hamiltonian function $H_{\e}^* =H^* + \e \bar H^*$ (where $H^*$ is defined by \eqref{eq_extendedunperturbedham} and $\bar H^*$ is the Hamiltonian of the perturbation, defined by \eqref{eq_hampert}), and consider the return map of the flow of $H_{\e}^*$ to the section $\{ \f = 0 \} \cap \{ H_{\e}^* = E_0 \}$ in a neighbourhood of $\Lambda^*$. Denote by $F$ the Poincar\'e map. Then $F$ has a normally hyperbolic invariant manifold $\hat \Lambda = \Lambda^* \cap \{ \f = 0 \} \cap \{ H_{\e}^* = E_0 \}$, and we denote by $f = F|_{\hat \Lambda}$ the restriction to $\hat \Lambda$ of the map $F$. Let us compute the map $f$. As a result of Lemma \ref{lemma_innersympform} we have 
\[
\left. \omega^* \right|_{\Lambda^*} = d \f \wedge dJ + d \theta \wedge d \Theta
\]
where the extended symplectic form $\omega^*$ is defined by \eqref{eq_extendedsymp}. Moreover, due to \eqref{eq_innervf}, and the fact that the Hamiltonian $\bar{H}^*$ of the perturbation is supported away from $\Lambda^*$, the vector field of the Hamiltonian $\left. H_{\e}^* \right|_{\Lambda^*}$ is $\dot \f = J, \, \dot J = 0, \dot \theta = 1, \, \dot \Theta = 0$. From this equation and the structure of the symplectic form $\left. \omega^* \right|_{\Lambda^*}$, it follows that the inner Hamiltonian is
\[
\left. H_{\e}^* \right|_{\Lambda^*} \left( \f, J, \theta, \Theta \right) = \frac{J^2}{2} + \Theta. 
\]
Therefore the energy level $\{ H_{\e}^* = E_0 \}$ in $\Lambda^*$ is defined by the equation $J=\sqrt{2 (E_0 - \Theta)}$, and the variables $\theta, \Theta$ define coordinates on the normally hyperbolic cylinder $\hat \Lambda = \Lambda^* \cap \{ \f = 0 \} \cap \{ H_{\e}^* = E_0 \}$. 

\begin{lemma}\label{lemma_innermap}
The inner map $f : \hat \Lambda \to \hat \Lambda$ is given by 
\begin{equation}\label{eq_innermap}
f : \begin{dcases}
\bar \theta =& \theta + \frac{2 \, \pi}{\sqrt{2 (E_0 - \Theta)}} \\
\bar \Theta =& \Theta
\end{dcases}
\end{equation}
and so $f$ satisfies a twist condition, with $\frac{\partial \bar \theta}{\partial \Theta} = O \left( J^{-3} \right)$. 
\end{lemma}

\begin{proof}
Denote by $\tau : \hat \Lambda \to \R$ the return time to $\hat \Lambda$ of the flow of the Hamiltonian $H_{\e}^*$. For a point $(\theta, \Theta) \in \hat \Lambda$ we write $J = J(\Theta) = \sqrt{2 (E_0 - \Theta)}$. Then 
\[
\tau (\theta, \Theta) = \frac{2 \, \pi}{J} = \frac{2 \, \pi}{\sqrt{2 (E_0 - \Theta)}}. 
\]
Since $\dot \theta = 1$ and $\dot \Theta = 0$ it follows that $f$ indeed has the form \eqref{eq_innermap}. Moreover
\[
\frac{\partial \bar \theta}{\partial \Theta} = \frac{2 \, \pi}{(2 (E_0 - \Theta))^{\frac{3}{2}}} \neq 0
\]
which implies the required twist condition. 
\end{proof}

Now, there is a foliation of the cylinder $\hat{\Lambda}$ by leaves
\[
\L \left( \Theta^* \right) = \left\{ (\theta, \Theta) \in \hat \Lambda : \Theta = \Theta^* \right\}
\]
for $\Theta^* \leq \Theta_0$, each of which is left invariant by the map $f$ as a consequence of Lemma \ref{lemma_innermap}.

Due to the contents of Section \ref{eq_scatteringmapcomp}, there is a scattering map $\hat S : \hat \Lambda \to \hat \Lambda'$ where $\hat \Lambda'$ is a locally invariant normally hyperbolic cylinder containing $\hat \Lambda$; moreover since $\left. \dot \Theta \right|_{\Lambda^*} = 0$ we have $\hat{S}(\theta, \Theta) = ( \bar \theta, \bar \Theta)$ where
\begin{equation}\label{eq_specificthetabar}
\bar \Theta = \Theta + \e \, J \, \left( \sin \left( \theta + \frac{a_- }{J} \right) - R_{\rho}^0 (0,J,\theta,\Theta) \right) + O \left( \e^2 J \right)
\end{equation}
and
\begin{equation}\label{eq_specificthetabarderivative}
\frac{\partial \bar \Theta}{\partial \theta} = - \e \, J \, \left( \cos \left( \theta + \frac{a_-}{J} \right) + R^1_{\rho} \left( 0, J , \theta, \Theta \right) \right) + O \left( \e^2 J \right)
\end{equation}
with $J=\sqrt{2 (E_0 - \Theta)}$, where we have set $\f = 0$ in \eqref{eq_thetascattering} and \eqref{eq_thetascatteringderivative}, and used \eqref{eq_chiiscos}. By Lemma \ref{lemma_scattering} we can choose $\rho > 0$ sufficiently small so that $\left| R_{\rho}^j (0,J,\theta,\Theta) \right| < \frac{1}{2 \sqrt{2}}$ for $j=0,1$ and for all $J \geq J_0$, $\theta \in T$, and $\Theta \leq \Theta_0$. Choose $\theta^*_{\pm} \in \T$ such that $\theta^*_{\pm} + \frac{a_-}{J} = \pm \frac{\pi}{4} \mod 2 \pi$. For any $\Theta^* \leq \Theta_0$, write $\left(\bar \theta^*_{\pm}, \bar \Theta^*_{\pm} \right)=\hat{S}(\theta^*_{\pm}, \Theta^*)$; by \eqref{eq_specificthetabar} and \eqref{eq_specificthetabarderivative} we have 
\begin{equation}\label{eq_transitionchainjumps}
\bar \Theta^*_+ - \Theta^* \geq \e \, J \, \frac{1}{2 \sqrt{2}} > 0, \quad \bar \Theta^*_- - \Theta^* \leq - \e \, J \, \frac{1}{2 \sqrt{2}} < 0, \quad \left| \frac{ \partial \bar \Theta^*_{\pm}}{\partial \theta} \right| \geq \e \, J \, \frac{1}{2 \sqrt{2}} > 0. 
\end{equation}
Equation \eqref{eq_transitionchainjumps} allows us to choose $\{ \Theta_n \}_{n \in \mathbb{N}}$ where either $\Theta_{n+1} - \Theta_n \geq \e \, J \, \frac{1}{2 \sqrt{2}}$ or $\Theta_{n+1} - \Theta_n \leq - \e \, J \, \frac{1}{2 \sqrt{2}}$, such that the leaves $\L \left( \Theta_n \right)$ and $\L \left( \Theta_{n+1} \right)$ are connected by the scattering map $\hat{S}$, in the sense that there is $z_n \in \L \left( \Theta_n \right)$ such that $\hat{S} ( z_n) \in \L \left( \Theta_{n+1} \right)$; moreover, $\hat S$ maps $\L \left( \Theta_n \right)$ transversely across $\L \left( \Theta_{n+1} \right)$ in the sense of the following lemma. 
\begin{lemma}\label{lemma_transversalityoffoliations}
Let $z_n = \left( \theta^*_{\alpha_n}, \Theta_n \right)$ where $\alpha_n \in \{ +,- \}$ for each $n$, and $\Theta_n$ is as described above, so that $z_n \in \L \left( \Theta_n \right)$ and $\hat S \left( z_n \right) \in \L \left( \Theta_n \right)$. Then we have
\begin{equation}\label{eq_transversalityoffoliations}
T_{\hat S (z_n)} \hat{\Lambda} = T_{\hat S (z_n)} \hat S \left( \L \left( \Theta_n \right) \right) + T_{\hat S (z_n)} \L \left( \Theta_{n+1} \right)
\end{equation}
as a result of the third inequality of \eqref{eq_transitionchainjumps}. 
\end{lemma}
\begin{proof}
Elements of the tangent space $T_{\hat S (z_n)} \hat{\Lambda} $ are of the form $(Q,P)$ where $Q$ is the coordinate pointing in the positive $\theta$ direction and $P$ is the coordinate pointing in the positive $\Theta$ direction. Moreover, elements of the tangent spaces $T_{z} \L \left( \Theta \right)$ are of the form $(Q,0)$ due to the topology of the leaves $\L \left( \Theta \right)$. Fix any $(Q,P) \in T_{\hat S (z_n)} \hat{\Lambda} \simeq \R^2$. We show that there is $(Q_0,0) \in T_{z_n} \L \left( \Theta_n \right)$ and $(Q_1,0) \in T_{\hat S \left( z_n \right)} \L \left( \Theta_{n+1} \right)$ such that $(Q,P) = D_{z_n} \hat S \, (Q_0,0) + (Q_1,0)$ which proves \eqref{eq_transversalityoffoliations}. Due to \eqref{eq_unperturbedscattering} and the fact that the Hamiltonian of the perturbation is independent of $\Theta$, the derivative $D \hat S$ is given by
\[
D \hat S = \left(
\begin{matrix}
1 & 0 \\
\frac{\partial \bar \Theta}{\partial \theta} & \frac{\partial \bar \Theta}{\partial \Theta}
\end{matrix}
\right). 
\]
Therefore we require
\[
\left( 
\begin{matrix}
Q \\
P
\end{matrix}
\right) = \left(
\begin{matrix}
1 & 0 \\
\frac{\partial \bar \Theta}{\partial \theta} & \frac{\partial \bar \Theta}{\partial \Theta}
\end{matrix}
\right) \left( 
\begin{matrix}
Q_0 \\
0
\end{matrix}
\right) + \left( 
\begin{matrix}
Q_1 \\
0
\end{matrix}
\right) = \left(
\begin{matrix}
Q_0 + Q_1 \\
\frac{\partial \bar \Theta}{\partial \theta} \, Q_0
\end{matrix}
\right). 
\]
Due to the third inequality of \eqref{eq_transitionchainjumps}, the derivative $\frac{\partial \bar \Theta}{\partial \theta}$ is bounded away from 0, and so we may choose $Q_0 = \left( \frac{\partial \bar \Theta}{\partial \theta} \right)^{-1} P$ and $Q_1 = Q - Q_0$ to complete the proof of the lemma. 
\end{proof}

Due to Lemmas \ref{lemma_innermap} and \ref{lemma_transversalityoffoliations}, the assumptions of the main theorem of \cite{gidea2006topological} are satisfied. It follows that\footnote{The statement of the main theorem in \cite{gidea2006topological} mentions only finite pseudo-orbits; however the theorem applies equally to infinite pseudo-orbits. Indeed, this has already been pointed out in \cite{clarke2022topological}; see also the discussion in \cite{clarke2022topological} regarding time estimates of the type \eqref{eq_timeestimates} when the system's parameters are nonuniform.}, for any $\eta > 0$, there are $\{ \hat z_n \}_{n \in \mathbb{N}} \subset \{ \f = 0 \} \cap \{ H^*_{\e} = E_0 \}$ and $\{ m_n \}_{n \in \mathbb{N}} \subset \mathbb{N}$ such that 
\[
\hat z_{n+1} = F^{m_n} (z_n), \quad d \left( \hat z_n, \L \left( \Theta_n \right) \right) < \eta. 
\]
Moreover the time for the flow to move a distance of order 1 in the $\Theta$ direction is of order
\begin{equation}\label{eq_timeestimates}
\min \left\{ \e^3 J^{-1}, \e \, J^{-4} \right\}. 
\end{equation}
This implies that when we are low on the cylinder (i.e. smaller values of $J = \sqrt{2 \left(E_0 - \Theta \right)}$) the time to move a distance of order 1 in the $\Theta$ direction is of order $\e^3 J^{-3}$, whereas when we are higher on the cylinder (i.e. larger values of $J$) the time to move a distance of order 1 in the $\Theta$ direction is of order $\e \, J^{-6}$. Equation \eqref{eq_timeestimates} follows from the time estimates given in \cite{clarke2022topological}; indeed, the order of the splitting of separatrices in the $\Theta$ direction is $\e$, the twist condition is of order $J^{-3}$ by Lemma \ref{lemma_innermap}, the jumps in the scattering map in the $\Theta$ direction are of order $\e \, J$ by Lemma \ref{lemma_scattering}, and the return time to the section $\{ \f = 0 \} \cap \{ H^*_{\e} = E_0 \}$ is of order $J^{-1}$. This completes the proof of Theorem \ref{theorem_main2}.

\appendix

\section{The Scattering Map of a Normally Hyperbolic Invariant Manifold}

In this section we denote by $M$ a $C^r$ smooth manifold, and by $\phi^t : M \to M$ a smooth complete flow with $\left. \frac{d}{dt} \right|_{t=0} \phi^t = X$ where $X \in C^r (M, TM)$. Let $\Lambda \subset M$ be a compact $\phi^t$-invariant submanifold, possibly with boundary. By $\phi^t$-invariant we mean that $X$ is tangent to $\Lambda$, but that orbits can escape through the boundary (a concept sometimes referred to as \emph{local} invariance). 

\begin{definition}
We call $\Lambda$ a \emph{normally hyperbolic invariant manifold} for $\phi^t$ if there is $0 < \lambda < \mu^{-1}$, a positive constant $C$ and an invariant splitting of the tangent bundle
\begin{equation}
T_{\Lambda} M = T \Lambda \oplus E^s \oplus E^u
\end{equation}
such that:
\begin{equation} \label{eq_normalhyperbolicity}
\def\arraystretch{1.5}
\begin{array}{c}
\| D \phi^t |_{E^s} \| \leq C \lambda^t \mbox{ for all } t \geq 0, \\
\| D \phi^t |_{E^{u}} \| \leq C \lambda^{-t} \mbox{ for all } t \leq 0, \\
\| D \phi^t |_{T \Lambda} \| \leq C \mu^{| t |} \mbox{ for all } t \in \mathbb{R}.
\end{array}
\end{equation}
Moreover, $\Lambda$ is called an $r$-\emph{normally hyperbolic invariant manifold} if it is $C^r$ smooth, and
\begin{equation}\label{eq_largespectralgap}
0 < \lambda < \mu^{-r} < 1
\end{equation}
 for $r \geq 1$. This is called a \emph{large spectral gap} condition.
\end{definition}

This definition guarantees the existence of stable and unstable invariant manifolds $W^{s,u} (\Lambda)\subset M$ defined as follows. The local stable manifold $W^{s}_{\mathrm{loc}}(\Lambda)$ is the set of points in a small neighbourhood of $\Lambda$ whose forward orbits never leave the neighbourhood, and tend exponentially to $\Lambda$. The local unstable manifold $W^{u}_{\mathrm{loc}}(\Lambda)$ is the set of points in the neighbourhood whose backward orbtis stay in the neighbourhood and tend exponentially to $\Lambda$.  We then define
\begin{equation}
W^s(\Lambda) = \bigcup_{t \geq 0}^{\infty} \phi^{-t} \left( W^{s}_{\mathrm{loc}}(\Lambda) \right), \quad W^u(\Lambda) = \bigcup_{t \geq 0}^{\infty} \phi^{t} \left( W^{u}_{\mathrm{loc}}(\Lambda) \right).
\end{equation}
On the stable and unstable manifolds we have the strong stable and strong unstable foliations, the leaves of which we denote by $W^{s,u}(x)$ for $x \in \Lambda$. For each $x \in \Lambda$, the leaf $W^s(x)$ of the strong stable foliation is tangent at $x$ to $E^s_x$, and the leaf $W^u(x)$ of the strong unstable foliation is tangent at $x$ to $E^u_x$. Moreover the foliations are invariant in the sense that $\phi^t \left( W^s (x) \right) = W^s \left( \phi^t (x) \right)$ and $\phi^t \left( W^u (x) \right) = W^u \left( \phi^t (x) \right)$ for each $x \in \Lambda$ and $t \in \mathbb{R}$. We thus define the \emph{holonomy maps} $\pi^{s,u} : W^{s,u} (\Lambda) \to \Lambda$ to be projections along leaves of the strong stable and strong unstable foliations. That is to say, if $x \in W^s (\Lambda)$ then there is a unique $x_+ \in \Lambda$ such that $x \in W^s(x_+)$, and so $\pi^s(x)=x_+$. Similarly, if $x \in W^u (\Lambda)$ then there is a unique $x_- \in \Lambda$ such that $x \in W^u(x_-)$, in which case $\pi^u(x)=x_-$.

Now, suppose that $x \in \left(W^s(\Lambda) \pitchfork W^u (\Lambda)\right) \setminus \Lambda$ is a transverse homoclinic point such that $x \in W^s(x_+) \cap W^u(x_-)$. We say that the homoclinic intersection at $x$ is \emph{strongly transverse} if
\begin{equation}
\begin{split} \label{eq_strongtransversality}
T_x W^s (x_+) \oplus T_x \left( W^s(\Lambda) \cap W^u(\Lambda) \right) = T_x W^s (\Lambda), \\
T_x W^u (x_-) \oplus T_x \left( W^s(\Lambda) \cap W^u(\Lambda) \right) = T_x W^u (\Lambda).
\end{split}
\end{equation}
In this case we can take a sufficiently small neighbourhood $\Gamma$ of $x$ in $W^s(\Lambda) \cap W^u(\Lambda)$ so that \eqref{eq_strongtransversality} holds at each point of $\Gamma$, and the restrictions to $\Gamma$ of the holonomy maps are bijections onto their images. We call $\Gamma$ a \emph{homoclinic channel} . We can then define the scattering map as follows \cite{delshams2008geometric}.

\begin{definition}
Let $y_-\in \pi^u \left( \Gamma \right)$, let $y = \left(\left. \pi^u \right|_{\Gamma} \right)^{-1} (y_-)$, and let $y_+ = \pi^s(y)$. The \emph{scattering map} $S : \pi^u (\Gamma) \to \pi^s (\Gamma)$ is defined by
\begin{equation}
S = \pi^s \circ \left( \pi^u \right)^{-1} : y_- \longmapsto y_+.
\end{equation}
\end{definition}

Suppose now that the smoothness $r$ of $M$ and $X$ is at least $2$, suppose the normally hyperbolic invariant manifold $\Lambda$ is a $C^r$ submanifold of $M$, and suppose the large spectral gap condition \eqref{eq_largespectralgap} holds. This implies $C^{r-1}$ regularity of the strong stable and strong unstable foliations \cite{hirsch1970invariant}, which in turn implies that the scattering map $S$ is $C^{r-1}$ \cite{delshams2008geometric}.

\bibliographystyle{abbrv}
\bibliography{gf_fa_refs} 

\end{document}